\documentclass{amsart}
\usepackage{amsmath}
\usepackage{amscd}
\usepackage{amssymb}
\usepackage{amsfonts}
\newtheorem{theorem}{Theorem}[section]
\newtheorem{lemma}[theorem]{Lemma}
\newtheorem{corollary}[theorem]{Corollary}

\theoremstyle{definition}
\newtheorem{definition}[theorem]{Definition}

\theoremstyle{remark}
\newtheorem{remark}[theorem]{Remark}
\numberwithin{equation}{section}

\begin{document}
\title[A nonlinear diffusion equation on complete manifolds]
{Gradient estimates for a nonlinear diffusion equation on complete
manifolds}
\author{Jia-Yong Wu}
\address{Department of Mathematics, East China Normal
University, Shanghai, China 200241}
 \email{jywu81@yahoo.com}

\date{January 1, 2009.}


\thanks{}

\begin{abstract}
Let $(M,g)$ be a complete non-compact Riemannian manifold with the
$m$-dimensional Bakry-\'{E}mery Ricci curvature bounded below by a
non-positive constant. In this paper, we give a localized
Hamilton-type gradient estimate for the positive smooth bounded
solutions to the following nonlinear diffusion equation
\[
u_t=\Delta u-\nabla\phi\cdot\nabla u-au\log u-bu,
\]
where $\phi$ is a $C^2$ function, and $a\neq0$ and $b$ are two real
constants. This work generalizes the results of Souplet and Zhang
(Bull. London Math. Soc., 38 (2006), pp. 1045-1053) and Wu
(Preprint, 2008).
\end{abstract}
\maketitle

\section{Introduction}\label{Int1}
Let $(M,g)$ be an $n$-dimensional non-compact Riemannian manifold
with the $m$-dimensional Bakry-\'{E}mery Ricci curvature bounded
below. Consider the following diffusion equation:
\begin{equation}\label{equ1}
u_t=\Delta u-\nabla \phi\cdot\nabla u-a u\log u- b u
\end{equation}
in $B(x_0,R)\times[t_0-T,t_0]\subset M\times(-\infty,\infty)$, where
$\phi$ is a $C^2$ function, and $a\neq0$ and $b$ are two real
constants. Eq. (\ref{equ1}) is closely linked with the gradient
Ricci solitons, which are the self-similar solutions to the Ricci
flow introduced by Hamilton \cite{[Hamilton]}. Ricci solitons have
inspired the entropy and Harnack estimates, the space-time
formulation of the Ricci flow, and the reduced distance and reduced
volume.

Below we recall the definition of Ricci solitons (see also Chapter 4
of \cite{[CLN]}).
\begin{definition}
A Riemannian manifold $(M,g)$ is called a \emph{gradient Ricci
soliton} if there exists a smooth function $f: M\rightarrow
\mathbb{R}$, sometimes called \emph{potential function}, such that
for some constant $c\in\mathbb{R}$, it satisfies
\begin{equation}\label{soliton}
Ric(g)+\nabla^{g}\nabla^{g} f=cg
\end{equation}
on $M$, where $Ric(g)$ is the Ricci curvature of manifold $M$ and
$\nabla^{g}\nabla^{g} f$ is the Hessian of $f$. A soliton is said to
be \emph{shrinking}, \emph{steady} or \emph{expanding} if the
constant $c$ is respectively positive, zero or negative.
\end{definition}

Suppose that $(M,g)$ be a gradient Ricci soliton, and $c$, $f$ are
described in Definition A. Letting $u=e^f$, under some curvature
assumptions, we can derive from (\ref{soliton}) that (cf.
\cite{[Ma]}, Eq. (7))
\begin{equation}\label{yinyong}
\Delta u+ 2cu \log u=(A_0-nc)u,
\end{equation}
for some constant $A_0$. Eq. (\ref{yinyong}) is a nonlinear elliptic
equation and a special case of Eq. (\ref{equ1}). For this kind of
equations, Ma (see Theorem 1 in \cite{[Ma]}) obtained the following
result.

\vspace{0.5em}

\noindent \textbf{Theorem A} (\cite{[Ma]}). \emph{Let $(M,g)$ be a
complete non-compact Riemannian manifold of dimension $n\geq3$ with
Ricci curvature bounded below by the constant $-K:=-K(2R)$, where
$R>0$ and $K(2R)\geq0$, in the metric ball $B_{2R}(p)$. Let $u$ be a
positive smooth solution to the elliptic equation
\begin{equation}\label{ellp}
\Delta u-au \log u=0
\end{equation}
with $a>0$. Let $f=\log u$ and let $(f, 2f)$ be the maximum among
$f$ and $2f$. Then there are two uniform positive constant $c_1$ and
$c_2$ such that
\begin{equation}
\begin{aligned}\label{Li Ma}
&|\nabla f|^2-a (f,2f)\\
\leq&
\frac{n\Big[(n+2)c^2_1+(n-1)c^2_1(1+R\sqrt{K})+c_2\Big]}{R^2}+2n\Big(|a|+K\Big)
\end{aligned}
\end{equation}
in $B_R(p)$.}

\vspace{0.5em}

Then Yang (see Theorem 1.1 in \cite{[Yang]}) extended the above
result and obtained the following local gradient estimate for the
nonlinear equation (\ref{equ1}) with $\phi\equiv c_0$, where $c_0$
is a fixed constant.

\vspace{0.5em}

\noindent \textbf{Theorem B} (\cite{[Yang]}). \emph{Let $M$ be an
$n$-dimensional complete non-compact Riemannian Manifold. Suppose
the Ricci curvature of $M$ is bounded below by $-K:=-K(2R)$, where
$R>0$ and $K(2R)\geq0$, in the metric ball $B_{2R}(p)$. If $u$ is a
positive smooth solution to Eq. (\ref{equ1}) with $\phi\equiv c_0$
on $M\times[0,\infty)$ and $f=\log u$, then for any $\alpha>1$ and
$0<\delta<1$
\begin{equation}
\begin{aligned}\label{Yang}
&|\nabla f|^2(x,t)-\alpha af(x,t)-\alpha b-\alpha f_t(x,t)\\
\leq& \frac{n\alpha^2}{2\delta
t}+\frac{n\alpha^2}{2\delta}\Bigg\{\frac{2\epsilon^2}{R^2}+\frac{\nu}{R^2}
+\sigma+\frac{\epsilon^2}{R^2}(n-1)\left(1+R\sqrt{K(2R)}\right)\\
&+\frac{K(2R)}{\alpha-1}+\frac{n\alpha^2\epsilon^2}{8(1-\delta)(\alpha-1)R^2}\Bigg\}
\end{aligned}
\end{equation}
in $B_R(p)\times(0,\infty)$, where $\epsilon>0$ and $\nu>0$ are some
constants and where $\sigma=a/2$ if $a>0$; $\sigma=-a$ if $a<0$.}

\vspace{0.5em}

Recently, the author (see Theorem 1.1 in \cite{[wu]}) used
Souplet-Zhang's method in \cite{[Sou-Zh]} and obtained a localized
Hamilton-type gradient estimate for the positive smooth bounded
solutions of the equation (\ref{equ1}) with $\phi\equiv c_0$.

\vspace{0.5em}

\noindent \textbf{Theorem C} (\cite{[wu]}). \emph{Let $(M,g)$ be an
$n$-dimensional non-compact Riemannian manifold with $Ric(M)\geq-K$
for some constant $K\geq0$. Suppose that $u(x,t)$ is a positive
smooth solution to the parabolic equation ({\ref{equ1}}) with
$\phi\equiv c_0$ in $Q_{R,T}\equiv B(x_0,R)\times[t_0-T,t_0]\subset
M\times(-\infty,\infty)$. Let $f:=\log u$. We also assume that there
exists non-negative constants $\alpha$ and $\delta$ such that
$\alpha-f \geq\delta>0$. Then there exist three dimensional
constants $\tilde{c}$,  $c(\delta)$ and $c(\alpha,\delta)$ such that
\begin{equation}\label{theor}
\frac{|\nabla u|}{u}\leq
\left(\frac{\tilde{c}}{R}\beta{+}\frac{c(\alpha,\delta)}{R}
{+}\frac{c(\delta)}{\sqrt{T}}{+}c(\delta)\left(|a|+K\right)^{1/2}\kern-2pt
{+}c(\delta)|a|^{1/2}\beta^{1/2}\right)\left(\alpha{-}\frac
ba{-}\log u\right)
\end{equation}
in $Q_{R/2, T/2}$, where
$\beta:=\max\left\{1,|\alpha/\delta-1|\right\}$.}

\vspace{0.5em}

The purpose of this paper is to extend Theorem C to the general
nonlinear diffusion equation (\ref{equ1}) via the $m$-dimensional
Bakry-\'{E}mery Ricci curvature.

Let us first  recall some facts about the $m$-dimensional
Bakry-\'{E}mery Ricci curvature (please see \cite{[Ba],[BE],
[BE2],[LD]} for more details). Given an $n$-dimensional Riemannian
manifold $(M, g)$ and a $C^2$ function $\phi$, we may define a
symmetric diffusion operator $ L:=\Delta-\nabla\phi\cdot\nabla,$
which is the infinitesimal generator of the Dirichlet form
\[
\mathcal {E}(f,g)=\int_M (\nabla f, \nabla g)\mathrm{d}\mu,
\,\,\,\forall f, g\in C_0^{\infty}(M),
\]
where $\mu$ is an invariant measure of $L$ given by $
\mathrm{d}\mu=e^{-\phi}\mathrm{d}x. $ It is well-known that $L$ is
self-adjoint with respect to the weighted measure $\mathrm{d}\mu$.

The $\infty$-dimensional Bakry-\'{E}mery Ricci curvature $Ric(L)$ is
defined by
\[
Ric(L):=Ric +Hess(\phi),
\]
where $Ric$ and $Hess$ denote the Ricci curvature of the metric $g$
and the Hessian respectively. Following the notation used in
\cite{[LD]}, we also define the $m$-dimensional Bakry-\'{E}mery
Ricci curvature of $L$ on an $n$-dimensional Riemaniann manifold as
follows
\[
Ric_{m,n}(L):= Ric(L)-\frac{\nabla \phi \otimes \nabla\phi}{m-n},
\]
where $m := \mathrm{dim}_{BE}(L)$ is called the Bakry-\'{E}mery
dimension of $L$. Note that the number $m$ is not necessarily to be
an integer and $m \geq n=\mathrm{dim} M$.

The main result of this paper can be stated in the following:
\begin{theorem}\label{main}
Let $(M,g)$ be an n-dimensional non-compact Riemannian manifold with
$Ric_{m,n}(L)\geq-K$ for some constant $K\geq0$. Suppose that
$u(x,t)$ is a positive smooth solution to the diffusion equation
({\ref{equ1}}) in $Q_{R,T}\equiv B(x_0,R)\times[t_0-T,t_0]\subset
M\times(-\infty,\infty)$. Let $f:=\log u$. We also assume that there
exists non-negative constants $\alpha$ and $\delta$ such that
$\alpha-f \geq\delta>0$. Then there exist three dimensional
constants $\tilde{c}$,  $c(\delta)$ and $c(\alpha,\delta,m)$ such
that
\begin{equation}\label{theor2}
\frac{|\nabla u|}{u}\leq
\left(\frac{\tilde{c}}{R}\beta{+}\frac{c(\alpha,\delta,m)}{R}
{+}\frac{c(\delta)}{\sqrt{T}}{+}c(\delta)\left(|a|+K\right)^{1/2}\kern-3pt
{+}c(\delta)|a|^{1/2}\beta^{1/2}\right)\left(\alpha{-}\frac
ba{-}\log u\right)
\end{equation}
in $Q_{R/2, T/2}$, where
$\beta:=\max\left\{1,|\alpha/\delta-1|\right\}$.
\end{theorem}
We make some remarks on the above theorem below.
\begin{remark}
(i). In Theorem \ref{main}, it seems that the assumption $\alpha-f
\geq\delta>0$ is reasonable. Because from this assumption, we can
get $u\leq e^{\alpha-\delta}$. We say that this upper bound of $u$
can be achieved in some setting. For example, from Corollary 1.2 in
\cite{[Yang]}, we know that positive smooth solutions to the
elliptic equation (\ref{ellp}) with $a<0$ have $u(x)\leq e^{n/2}$
for all $x\in M$ provided the Ricci curvature of $M$ is
non-negative.

(ii). Note that the theorem still holds if $m$-dimensional
Bakry-\'{E}mery Ricci curvature is replaced by $\infty$-dimensional
Bakry-\'{E}mery Ricci curvature. In fact this result can be obtained
by (\ref{gongshi}) in Section \ref{sec2}.

(iii). Theorem \ref{main} generalizes the above mentioned Theorem C.
When we choose $\phi\equiv c_0$, we return Theorem C. The proof of
our main theorem is based on Souplet-Zhang's gradient estimate and
the trick used in \cite{[wu]} with some modifications.
\end{remark}

In particular, if $u(x,t)\leq1$ is a positive smooth solution to the
diffusion equation ({\ref{equ1}}) with $a<0$, then we have a simple
estimate.
\begin{corollary}\label{Coro}
Let $(M,g)$ be an n-dimensional non-compact Riemannian manifold with
$Ric_{m,n}(L)\geq-K$ for some constant $K\geq0$. Suppose that
$u(x,t)\leq 1$ is a positive smooth solution to the diffusion
equation ({\ref{equ1}}) with $a<0$ in $Q_{R,T}\equiv
B(x_0,R)\times[t_0-T,t_0]\subset M\times(-\infty,\infty)$. Then
there exist two dimensional constants $c$ and $c(m)$ such that
\begin{equation}\label{cor}
\frac{|\nabla u|}{u}\leq \left(\frac
{c(m)}{R}+\frac{c}{\sqrt{T}}+c\sqrt{K+|a|}\right)\left(1-\frac
ba+\log{\frac1u}\right)
\end{equation}
in $Q_{R/2, T/2}$.
\end{corollary}

\begin{remark}
We point out that our localized Hamilton-type gradient estimate can
be also regarded as the generalization of the result of
Souplet-Zhang \cite{[Sou-Zh]} for the heat equation on complete
manifolds. In fact, the above Corollary \ref{Coro} is similar to the
result of Souplet-Zhang (see Theorem 1.1 of \cite{[Sou-Zh]}). From
the inequality (\ref{zlihou}) below, we can conclude that if
$\phi\equiv c_0$ and $a=0$, then our result can be reduced to
theirs.
\end{remark}

The method of proving Theorem \ref{main} is the gradient estimate,
which is originated by Yau \cite{[Yau]} (see also Cheng-Yau
\cite{[Cheng-Yau]}), and developed further by Li-Yau
\cite{[Li-Yau]}, Li \cite{[LJY]} and Negrin \cite{[Negrin]}. Then R.
S. Hamilton \cite{[Hamilton2]} gave an elliptic type gradient
estimate for the heat equation. But this type of estimate is a
global result which requires the heat equation defined on closed
manifolds. Recently, a localized Hamilton-type gradient estimate was
proved by Souplet and Zhang \cite{[Sou-Zh]}, which can be viewed as
a combination of Li-Yau's Harnack inequality \cite{[Li-Yau]} and
Hamilton's gradient estimate \cite{[Hamilton2]}. In this paper, we
obtain a localized Hamilton-type gradient estimate for a general
diffusion equation ({\ref{equ1}}) as Souplet and Zhang in
\cite{[Sou-Zh]} did for the heat equation on complete manifolds. To
prove Theorem \ref{main}, we mainly follow the arguments of
Souplet-Zhang in \cite{[Sou-Zh]}, together with some facts about
Bakry-\'{E}mery Ricci curvature. Note that the diffusion equation
(\ref{equ1}) is nonlinear. So our case is a little more complicated
than theirs.

The structure of this paper is as follows. In Section \ref{sec2}, we
will give a basic lemma to prepare for proving Theorem \ref{main}.
Section \ref{sec3} is devoted to the proof of Theorem \ref{main}. In
Section \ref{Sec4}, we will prove Corollary \ref{Coro} in the case
$0<u\leq1$ with $a<0$.

\section{A basic lemma}\label{sec2}
In this section, we will prove the following lemma which is
essential in the derivation of the gradient estimate of the equation
(\ref{equ1}). Replacing $u$ by $e^{-b/a}u$, we only need to consider
positive smooth solutions of the following diffusion equation:

\begin{equation}\label{maineq}
u_t=\Delta u-\nabla \phi\cdot\nabla u-a u\log u.
\end{equation}
Suppose that $u(x,t)$ is a positive smooth solution to the diffusion
equation (\ref{equ1}) in $Q_{R,T}\equiv B(x_0,R)\times[t_0-T,t_0]$.
Define a smooth function
\[
f(x,t):=\log u(x,t)
\]
in $Q_{R,T}$. By (\ref{maineq}), we have
\begin{equation}\label{lemequ}
\left(L-\frac{\partial}{\partial t}\right)f+|\nabla f|^2-a f=0.
\end{equation}
Then we have the following lemma, which is a generalization of the
computation carried out in \cite{[Sou-Zh],[wu]}.
\begin{lemma}\label{Le11}
Let $(M,g)$ be an n-dimensional non-compact Riemannian manifold with
$Ric_{m,n}(L)\geq-K$ for some constant $K\geq0$. Let $f(x,t)$ is a
smooth function defined on $Q_{R,T}$ satisfying the diffusion
equation (\ref{lemequ}). We also assume that there exist
non-negative constants $\alpha$ and $\delta$ such that $\alpha-f
\geq\delta>0$. Then for all $(x,t)$ in $Q_{R,T}$ the function
\begin{equation}\label{lemmaequ2}
\omega:=\left|\nabla\log(\alpha-f)\right|^2=\frac{|\nabla
f|^2}{(\alpha-f)^2}
\end{equation}
satisfies the following inequality
\begin{equation}
\begin{aligned}\label{lemmaequ3}
&\left(L-\frac{\partial}{\partial
t}\right)\omega\\
\geq&\frac{2(1-\alpha)+2f}{\alpha-f}\left\langle \nabla
f,\nabla\omega\right\rangle
+2(\alpha-f)\omega^2+2(a-K)\omega+\frac{2af}{\alpha-f}\omega.
\end{aligned}
\end{equation}
\end{lemma}
\begin{proof}
By (\ref{lemmaequ2}), we have
\begin{equation}
\begin{aligned}\label{lemmpro1}
\omega_j=\frac{2f_i f_{ij}}{(\alpha-f)^2}+\frac{2f^2_i
f_j}{(\alpha-f)^3},
\end{aligned}
\end{equation}

\begin{equation}
\begin{aligned}\label{dgemf}
\Delta\omega=\frac{2|f_{ij}|^2}{(\alpha-f)^2}+\frac{2f_i
f_{ijj}}{(\alpha-f)^2}+\frac{8f_if_j
f_{ij}}{(\alpha-f)^3}+\frac{2f^2_i
f_{jj}}{(\alpha-f)^3}+\frac{6f^2_i f^2_j}{(\alpha-f)^4}
\end{aligned}
\end{equation}
and
\begin{equation}
\begin{aligned}\label{bake-em}
L\omega&=\Delta\omega-\phi_j\omega_j\\
&=\frac{2|f_{ij}|^2}{(\alpha{-}f)^2}+\frac{2f_i
f_{ijj}}{(\alpha{-}f)^2}+\frac{8f_if_j f_{ij}}{(\alpha{-}f)^3}
+\frac{2f^2_i f_{jj}}{(\alpha{-}f)^3}+\frac{6f^4_i}{(\alpha{-}f)^4}
-\frac{2f_{ij}f_i\phi_j}{(\alpha{-}f)^2}
-\frac{2f^2_if_j\phi_j}{(\alpha{-}f)^3}\\
&=\frac{2|f_{ij}|^2}{(\alpha{-}f)^2}+\frac{2f_i
(Lf)_i}{(\alpha{-}f)^2}+\frac{2(R_{ij}+\phi_{ij})f_if_j}{(\alpha{-}f)^2}+\frac{8f_if_j
f_{ij}}{(\alpha{-}f)^3} +\frac{2f^2_i\cdot
Lf}{(\alpha{-}f)^3}+\frac{6f^4_i}{(\alpha{-}f)^4},
\end{aligned}
\end{equation}
where $f_i:=\nabla_i f$ and $f_{ijj}:=\nabla_j\nabla_j\nabla_i f$,
etc. By (\ref{lemmaequ2}) and (\ref{lemequ}), we also have
\begin{equation}
\begin{aligned}\label{derivatt}
\omega_t&=\frac{2\nabla_if\cdot\nabla_i\Big[Lf+|\nabla
f|^2-af\Big]}{(\alpha-f)^2}+\frac{2|\nabla f|^2\Big[Lf+|\nabla
f|^2-af\Big]}{(\alpha-f)^3}\\
&=\frac{2\nabla f \nabla L
f}{(\alpha-f)^2}+\frac{4f_if_jf_{ij}}{(\alpha-f)^2}-\frac{2a|\nabla
f|^2}{(\alpha-f)^2}+\frac{2f^2_iL f}{(\alpha-f)^3}+\frac{2|\nabla
f|^4}{(\alpha-f)^3}-\frac{2af|\nabla f|^2}{(\alpha-f)^3}.
\end{aligned}
\end{equation}
Combining ({\ref{bake-em}}) with ({\ref{derivatt}}), we can get
\begin{equation}
\begin{aligned}\label{dge}
\left(L-\frac{\partial}{\partial
t}\right)\omega&=\frac{2|f_{ij}|^2}{(\alpha-f)^2}
+\frac{2(R_{ij}+\phi_{ij})f_if_j}{(\alpha-f)^2}+\frac{8f_if_j
f_{ij}}{(\alpha-f)^3}+\frac{6f^4_i}{(\alpha-f)^4}\\
&\,\,\,\,\,\,-\frac{4f_if_j
f_{ij}}{(\alpha-f)^2}-\frac{2f^4_i}{(\alpha-f)^3}
+\frac{2af^2_i}{(\alpha-f)^2}+\frac{2aff^2_i}{(\alpha-f)^3}.
\end{aligned}
\end{equation}
Noting that $Ric_{m,n}(L)\geq-K$ for some constant $K\geq0$, we have
\begin{equation}\label{gongshi}
(R_{ij}+\phi_{ij})f_if_j\geq\frac{|\nabla \phi\cdot\nabla
f|^2}{m-n}-K|\nabla f|^2\geq-K|\nabla f|^2.
\end{equation}
By (\ref{lemmpro1}),  we have
\begin{equation}
\begin{aligned}\label{ledcv}
\omega_jf_j=\frac{2f_if_j f_{ij}}{(\alpha-f)^2}+\frac{2f^2_i
f^2_j}{(\alpha-f)^3},
\end{aligned}
\end{equation}
and consequently,
\begin{equation}
\begin{aligned}\label{ledml}
0=-2\omega_jf_j+\frac{4f_if_j f_{ij}}{(\alpha-f)^2}+\frac{4f^4_i
}{(\alpha-f)^3},
\end{aligned}
\end{equation}
\begin{equation}
\begin{aligned}\label{levnmgf}
0=\frac{1}{\alpha-f}\left[2\omega_jf_j-\frac{4f^4_i}{(\alpha-f)^3}\right]
-\frac{4f_if_jf_{ij}}{(\alpha-f)^3}.
\end{aligned}
\end{equation}
Substituting (\ref{gongshi}) into (\ref{dge}) and then adding
(\ref{dge}) with (\ref{ledml}) and (\ref{levnmgf}), we can get
\begin{equation}
\begin{aligned}\label{dgenhjfy}
\left(L-\frac{\partial}{\partial
t}\right)\omega&\geq\frac{2|f_{ij}|^2}{(\alpha-f)^2}-\frac{2K|\nabla
f|^2}{(\alpha-f)^2}+\frac{4f_if_j
f_{ij}}{(\alpha-f)^3}+\frac{2f^4_i}{(\alpha-f)^4}+\frac{2f^4_i}{(\alpha-f)^3}\\
&\,\,\,\,\,\,+\frac{2(1-\alpha)+2f}{\alpha-f}f_i \omega_i
+\frac{2af^2_i}{(\alpha-f)^2}+\frac{2aff^2_i}{(\alpha-f)^3}.
\end{aligned}
\end{equation}
Note that $\alpha-f\geq\delta>0$ implies
\begin{equation*}
\begin{aligned}
\frac{2|f_{ij}|^2}{(\alpha-f)^2}+\frac{4f_if_j
f_{ij}}{(\alpha-f)^3}+\frac{2f^4_i}{(\alpha-f)^4}\geq 0.
\end{aligned}
\end{equation*}
This, together with (\ref{dgenhjfy}), yields the desired estimate
(\ref{lemmaequ3}).

\end{proof}
\section{Proof of Theorem \ref{main}}\label{sec3}
In this section, we will use Lemma \ref{Le11} and the localization
technique of Souplet-Zhang \cite{[Sou-Zh]} to give the elliptic type
gradient estimates on the positive and bounded smooth solutions of
the diffusion equation ({\ref{equ1}}).
\begin{proof}
First we give the well-known cut-off function by Li-Yau
\cite{[Li-Yau]} (see also \cite{[Sou-Zh]}) as follows. We caution
the reader that the calculation is not the same as that in
\cite{[Li-Yau]} due to the difference of the first-order term.

Let $\psi=\psi(x,t)$ be a smooth cut-off function supported in
$Q_{R,T}$ satisfying the following properties:
\begin{enumerate}
\item[(1)]$\psi=\psi(d(x,x_0),t)\equiv \psi(r,t)$; $\psi(x,t)=1$ in
$Q_{R/2,T/2}$, $0\leq\psi\leq1$;
\end{enumerate}

\begin{enumerate}
\item[(2)]$\psi$ is decreasing as a radial function in the spatial
variables;
\end{enumerate}

\begin{enumerate}
\item[(3)]$\frac{|\partial_r
\psi|}{\psi^\epsilon}\leq\frac{C_\epsilon}{R}$, $\frac{|\partial^2_r
\psi|}{\psi^\epsilon}\leq\frac{C_\epsilon}{R^2}$, when
$0<\epsilon<1$;
\end{enumerate}

\begin{enumerate}
\item[(4)]$\frac{|\partial_t \psi|}{\psi^{1/2}}\leq\frac{C}{T}$.
\end{enumerate}

From Lemma \ref{Le11}, by a straight forward calculation, we have
\begin{equation}
\begin{aligned}\label{lemdx3}
&L(\psi\omega)-\frac{2(1-\alpha)+2f}{\alpha-f}\nabla
f\cdot\nabla(\psi\omega) -2\frac{\nabla\psi}{\psi}
\cdot\nabla(\psi\omega)-(\psi\omega)_t\\
\geq&2\psi(\alpha-f)\omega^2
-\left[\frac{2(1-\alpha)+2f}{\alpha-f}\nabla f\cdot\nabla\psi\right]
\omega-2\frac{|\nabla\psi|^2}{\psi}\omega\\
&+(L\psi)\omega-\psi_t\omega
+2(a-K)\psi\omega+2\frac{af}{\alpha-f}\psi\omega.
\end{aligned}
\end{equation}
Let $(x_1,t_1)$ be a point where $\psi\omega$  achieves the maximum.
By Li-Yau \cite{[Li-Yau]}, without loss of generality we assume that
$x_1$ is not in the cut-locus of $M$. Then at this point, we have
\begin{equation*}
\begin{aligned}
L(\psi\omega)\leq0,\,\,\,\,\,\,(\psi\omega)_t\geq0,
\,\,\,\,\,\,\nabla(\psi\omega)=0.
\end{aligned}
\end{equation*}
Hence at $(x_1,t_1)$, by (\ref{lemdx3}), we get
\begin{equation}
\begin{aligned}\label{lefor}
2\psi(\alpha-f)\omega^2(x_1,t_1)&\leq
\Bigg\{\left[\frac{2(1-\alpha)+2f}{\alpha-f}\nabla
f\cdot\nabla\psi\right]\omega
+2\frac{|\nabla\psi|^2}{\psi}\omega-(L\psi)\omega\\
&\,\,\,\,\,\,\,\,\,+\psi_t\omega-2(a-K)\psi\omega
-2\frac{af}{\alpha-f}\psi\omega\Bigg\}(x_1,t_1).
\end{aligned}
\end{equation}
In the following, we will introduce the upper bounds for each term
of the right-hand side (RHS) of (\ref{lefor}). Following similar
arguments of Souplet-Zhang (\cite{[Sou-Zh]}, pp. 1050-1051), we have
the estimates of the first term of the BHS of (\ref{lefor})
\begin{equation}
\begin{aligned}\label{term1}
&\left[\frac{2f}{\alpha-f}\nabla f\cdot\nabla\psi\right]\omega\\
\leq&2|f|\cdot|\nabla\psi|\cdot\omega^{3/2}=2\left[\psi(\alpha-f)\omega^2\right]^{3/4}
\cdot\frac{|f|\cdot|\nabla\psi|}{[\psi(\alpha-f)]^{3/4}}\\
\leq&\psi(\alpha-f)\omega^2+\tilde{c}
\frac{(f|\nabla\psi|)^4}{[\psi(\alpha-f)]^3}
\leq\psi(\alpha-f)\omega^2+\tilde{c}\frac{f^4}{R^4(\alpha-f)^3}
\end{aligned}
\end{equation}
and
\begin{equation}
\begin{aligned}\label{tianjia}
&\left[\frac{2(1-\alpha)}{\alpha-f}\nabla f
\cdot\nabla\psi\right]\omega\\
\leq&2|1-\alpha||\nabla\psi|\omega^{3/2}
=(\psi\omega^2)^{3/4}\cdot\frac{2|1-\alpha||\nabla\psi|}{\psi^{3/4}}\\
\leq&\frac{\delta}{12}\psi\omega^2+c(\alpha,\delta)
\left(\frac{|\nabla\psi|}{\psi^{3/4}}\right)^4
\leq\frac{\delta}{12}\psi\omega^2+\frac{c(\alpha,\delta)}{R^4}.
\end{aligned}
\end{equation}
For the second term of the RHS of (\ref{lefor}), we have
\begin{equation}
\begin{aligned}\label{term2}
2\frac{|\nabla\psi|^2}{\psi}\omega
&=2\psi^{1/2}\omega\cdot\frac{|\nabla\psi|^2}{\psi^{3/2}}
\leq\frac{\delta}{12}\psi\omega^2+c(\delta)
\left(\frac{|\nabla\psi|^2}{\psi^{3/2}}\right)^2\\
&\leq\frac{\delta}{12}\psi\omega^2+\frac{c(\delta)}{R^4}.
\end{aligned}
\end{equation}
For the third term of the RHS of (\ref{lefor}), since
$Ric_{m,n}(L)\geq-K$, by the generalized Laplacian comparison
theorem (see \cite{[BE2]} or \cite{[LD]}),
\[
Lr\leq(m-1)\sqrt{K}\coth(\sqrt{K}r).
\]
Consequently, we have
\begin{equation}
\begin{aligned}\label{term3}
-(L\psi)\omega&=-\left[(\partial_r\psi)Lr+(\partial^2_r\psi)\cdot
|\nabla r|^2\right]\omega\\
&\leq-\left[\partial_r\psi(m-1)\sqrt{K}\coth(\sqrt{K}r)
+\partial^2_r\psi\right]\omega\\
&\leq-\left[\partial_r\psi(m-1)\left(\frac1r+\sqrt{K}\right)
+\partial^2_r\psi\right]\omega\\
&\leq\left[|\partial^2_r\psi|+2(m-1)\frac{|\partial_r\psi|}{R}
+(m-1)\sqrt{K}|\partial_r\psi|\right]\omega\\
&\leq\psi^{1/2}\omega\frac{|\partial^2_r\psi|}{\psi^{1/2}}+\psi^{1/2}
\omega2(m-1)\frac{|\partial_r\psi|}{R\psi^{1/2}}+\psi^{1/2}\omega
(m-1)\frac{\sqrt{K}|\partial_r\psi|}{\psi^{1/2}}\\
&\leq\frac{\delta}{12}\psi\omega^2+c(\delta,m)
\left[\left(\frac{|\partial^2_r\psi|}{\psi^{1/2}}\right)^2
+\left(\frac{|\partial_r\psi|}{R\psi^{1/2}}\right)^2
+\left(\frac{\sqrt{K}|\partial_r\psi|}{\psi^{1/2}}\right)^2\right]\\
&\leq\frac{\delta}{12}\psi\omega^2+\frac{c(\delta,m)}{R^4}
+\frac{c(\delta,m)K}{R^2}.
\end{aligned}
\end{equation}
Now we estimate the fourth term:
\begin{equation}
\begin{aligned}\label{term4}
|\psi_t|\omega&=\psi^{1/2}\omega\frac{|\psi_t|}{\psi^{1/2}}
\leq\frac{\delta}{12}\left(\psi^{1/2}\omega\right)^2+c(\delta)
\left(\frac{|\psi_t|}{\psi^{1/2}}\right)^2\\
&\leq\frac{\delta}{12}\psi\omega^2+\frac{c(\delta)}{T^2}.
\end{aligned}
\end{equation}
Notice that we have used Young's inequality below in obtaining
(\ref{term1})-(\ref{term4}):
\[
ab\leq\frac{a^p}{p}+\frac{b^q}{q},\,\,\,\,\,\, \forall\,\,\, p,q>0
\,\,\,\mathrm{with}\,\,\, \frac1p+\frac1q=1.
\]
Finally, we estimate the last two terms:
\begin{equation}
\begin{aligned}\label{term5}
-2(a-K)\psi\omega\leq2(|a|+K)\psi\omega\leq\frac{\delta}{12}\psi\omega^2+c(\delta)(|a|+K)^2;
\end{aligned}
\end{equation}
and
\begin{equation}\label{term6}
-2\frac{af}{\alpha-f}\psi\omega\leq2\frac{|a|\cdot|f|}{\alpha-f}\psi\omega
\leq\frac{\delta}{12}\psi\omega^2+c(\delta)a^2\frac{f^2}{(\alpha-f)^2}.
\end{equation}
Substituting (\ref{term1})-(\ref{term6}) to the RHS of (\ref{lefor})
at $(x_1,t_1)$, we get
\begin{equation}
\begin{aligned}\label{leforfor}
2\psi(\alpha-f)\omega^2&\leq \psi(\alpha-f)\omega^2
+\frac{\tilde{c}f^4}{R^4(\alpha-f)^3}+\frac{\delta}{2}\psi\omega^2
+\frac{c(\alpha,\delta)}{R^4}+\frac{c(\delta)}{R^4}+\frac{c(\delta,m)}{R^4}\\
&\,\,\,\,\,\, +\frac{c(\delta,m)K}{R^2}+\frac{c(\delta)}{T^2}
+c(\delta)(|a|+K)^2+c(\delta)a^2\frac{f^2}{(\alpha-f)^2}.
\end{aligned}
\end{equation}
Recall that $\alpha-f \geq\delta>0$, (\ref{leforfor}) implies
\begin{equation}
\begin{aligned}\label{lefogong}
\psi\omega^2(x_1,t_1)&\leq
\tilde{c}\frac{f^4}{R^4(\alpha-f)^4}+\frac{1}{2}\psi\omega^2(x_1,t_1)
+\frac{c(\alpha,\delta)}{R^4}+\frac{c(\delta,m)}{R^4}+\frac{c(\delta,m)K}{R^2}\\
&\,\,\,\,\,\,+\frac{c(\delta)}{T^2}
+c(\delta)(|a|+K)^2+c(\delta)a^2\frac{f^2}{(\alpha-f)^2}.
\end{aligned}
\end{equation}
Furthermore, we need to estimate the RHS of (\ref{lefogong}). If
$f\leq0$ and $\alpha\geq0$, then we have
\begin{equation}
\begin{aligned}\label{gujief1}
\frac{f^4}{(\alpha-f)^4}\leq
1,\,\,\,\,\,\,\,\,\,\,\,\,\frac{f^2}{(\alpha-f)^2}\leq 1;
\end{aligned}
\end{equation}
if $f>0$, by the assumption $\alpha-f \geq\delta>0$, we know that
\begin{equation}
\begin{aligned}\label{gujief2}
\frac{f^4}{(\alpha-f)^4}\leq
\frac{(\alpha-\delta)^4}{\delta^4}=\left(\frac{\alpha}{\delta}-1\right)^4,
\,\,\,\,\,\,\,\,\,\,\,\,\frac{f^2}{(\alpha-f)^2}\leq
\left(\frac{\alpha}{\delta}-1\right)^2.
\end{aligned}
\end{equation}
Plugging (\ref{gujief1}) (or (\ref{gujief2})) into (\ref{lefogong}),
we obtain
\begin{equation}
\begin{aligned}\label{lefogongdd}
(\psi\omega^2)(x_1,t_1)\leq
\frac{\tilde{c}\beta^4+c(\alpha,\delta,m)}{R^4}
+\frac{c(\delta,m)K}{R^2}+\frac{c(\delta)}{T^2}
+c(\delta)(|a|+K)^2+c(\delta)a^2\beta^2,
\end{aligned}
\end{equation}
where $\beta:=\max\left\{1,|\alpha/\delta-1|\right\}$.

\noindent The above inequality implies, for all $(x,t)$ in $Q_{R,T}$
\begin{equation}
\begin{aligned}
(\psi^2\omega^2)(x,t)&\leq\psi^2(x_1,t_1)\omega^2(x_1,t_1)
\leq\psi(x_1,t_1)\omega^2(x_1,t_1)\\
&\leq\frac{\tilde{c}\beta^4+c(\alpha,\delta,m)}{R^4}
+\frac{c(\delta,m)K}{R^2}+\frac{c(\delta)}{T^2}
+c(\delta)(|a|+K)^2+c(\delta)a^2\beta^2.
\end{aligned}
\end{equation}
Note that $\psi(x,t)=1$ in $Q_{R/2,T/2}$ and $\omega={|\nabla
f|^2}/{(\alpha-f)^2}$. Therefore we have
\begin{equation}
\begin{aligned}\label{conclus1.}
\frac{|\nabla f|}{\alpha-f}\leq
\left(\frac{\tilde{c}\beta^4+c(\alpha,\delta,m)}{R^4}
+\frac{c(\delta,m)K}{R^2}+\frac{c(\delta)}{T^2}
+c(\delta)(|a|+K)^2+c(\delta)a^2\beta^2\right)^{1/4}.
\end{aligned}
\end{equation}
Since $f=\log u$, we get the following estimate for Eq.
(\ref{maineq})
\begin{equation}
\begin{aligned}\label{conclus2.}
\frac{|\nabla u|}{u}\leq
\left(\frac{\tilde{c}\beta^4+c(\alpha,\delta,m)}{R^4}
+\frac{c(\delta)}{T^2} +c(\delta)(|a|+K)^2
+c(\delta)a^2\beta^2\right)^{1/4}\Big(\alpha-\log u\Big).
\end{aligned}
\end{equation}
Replacing $u$ by $e^{b/a}u$ gives the desired estimate
(\ref{theor2}). This completes the proof of Theorem \ref{main}.
\end{proof}

\section{Proof of Corollary \ref{Coro}}\label{Sec4}

\begin{proof}
The proof is similar to that of Theorem \ref{main}. We still use the
technique of a cut-off function in a local neighborhood of
Riemannian manifolds. For $0<u\leq1$, we let $f=\log u$. Then
$f\leq0$. Set
\[
\omega:=|\nabla\log(1-f)|^2=\frac{|\nabla f|^2}{(1-f)^2}.
\]
By Lemma \ref{Le11}, we have
\begin{equation}
\begin{aligned}\label{corrformu}
\left(L-\frac{\partial}{\partial t}\right)
\omega\geq\frac{2f}{1-f}\left\langle \nabla
f,\nabla\omega\right\rangle+2(1-f)\omega^2-2(|a|+K)\omega.
\end{aligned}
\end{equation}
We define a smooth cut-off function $\psi=\psi(x,t)$ in the same way
as Section \ref{sec3}. Follow all steps in the last section (see
also pp. 1050-1051 in \cite{[Sou-Zh]}), we can easily get the
following inequality
\begin{equation}
\begin{aligned}\label{jiezhi}
2(1-f)\psi\omega^2&\leq(1-f)\psi\omega^2+\frac{cf^4}{R^4(1-f)^3}
+\frac{\psi\omega^2}{2}+\frac{c}{R^4}\\
&\,\,\,\,\,\,+\frac{c(m)}{R^4}+\frac{c(m)K}{R^2}+\frac{c}{T^2}
+c(|a|+K)^2,
\end{aligned}
\end{equation}
where we used similar estimates (\ref{term1})-(\ref{term6}) with the
difference that these estimates do not contain the parameter
$\delta$. Using the same method as that in proving Theorem
\ref{main}, for all $(x,t)$ in $Q_{R/2, T/2}$ we can get
\begin{equation}
\begin{aligned}\label{jiezhi}
\omega^2(x,t)&\leq\frac{c(m)}{R^4}+\frac{c(m)K}{R^2}
+\frac{c}{T^2}+c(|a|+K)^2\\
&\leq\frac{c(m)}{R^4}+\frac{c(m)}{R^2}(|a|+K)+\frac{c}{T^2}+c(|a|+K)^2\\
&\leq\frac{c(m)}{R^4}+\frac{c}{T^2}+c(|a|+K)^2.
\end{aligned}
\end{equation}
Again, using the same argument in the proof of Theorem \ref{main}
gives
\begin{equation}\label{zlihou}
\frac{|\nabla f|}{1-f}\leq \frac
{c(m)}{R}+\frac{c}{\sqrt{T}}+c\sqrt{K+|a|},
\end{equation}
where $c$ is a constant depending only on $n$, $c(m)$ is a constant
depending only on $n$ and $m$.

Since $f=\log u$, we get
\begin{equation}\label{zlihou2}
\frac{|\nabla u|}{u}\leq\left(\frac
{c(m)}{R}+\frac{c}{\sqrt{T}}+c\sqrt{K+|a|}\right)\cdot\left(1+\log{\frac1u}\right).
\end{equation}
At last, replacing $u$ by $e^{b/a}u$ above yields (\ref{cor}).
\end{proof}

\section*{Acknowledgment}
The author would like to thank Professor Yu Zheng for his helpful
suggestions on this problem, and for his encouragement. He would
also like to thank the referees for useful suggestions. This work is
partially supported by the NSFC10871069.

\bibliographystyle{amsplain}

\end{document}